\documentclass[12pt,reqno]{amsart}
\usepackage{amsmath, amsfonts, amssymb, amsthm, hyperref, cite}
\usepackage{bm}
\allowdisplaybreaks[4]
\def\l{\left}
\def\r{\right}
\def\bg{\bigg}
\def\({\bg(}
\def\){\bg)}
\def\t{\text}
\def\f{\frac}

\def\eq{\equiv}

\def\Z{\mathbb Z}
\def\C{\mathbb C}
\def\N{\mathbb N}

\def\R{\mathbb R}

\def\e{\t{e}}
\def\<{\langle}
\def\>{\rangle}
\def\1{{\bf 1}}

\theoremstyle{plain}
\newtheorem{theorem}{Theorem}[section]

\newtheorem{lemma}{Lemma}[section]
\newtheorem{corollary}{Corollary}
\theoremstyle{definition}
\newtheorem*{Ack}{Acknowledgment}

\theoremstyle{remark}
\newtheorem{Rem}{Remark}[section]

\numberwithin{equation}{section}

\begin{document}
\title{Supercongruences involving binomial coefficients and Euler polynomials}
\author[Chen Wang]{Chen Wang*}
\address{Department of Applied Mathematics, Nanjing Forestry
University, Nanjing 210037, People's Republic of China}
\email{cwang@smail.nju.edu.cn}
\author[Hui-Li Han]{Hui-Li Han}
\address{Department of Applied Mathematics, Nanjing Forestry
University, Nanjing 210037, People's Republic of China}
\email{mintcrescent@163.com}

\begin{abstract}
Let $p$ be an odd prime and let $x$ be a $p$-adic integer. In this paper, we establish supercongruences for
$$
\sum_{k=0}^{p-1}\frac{\binom{x}{k}\binom{x+k}{k}(-4)^k}{(dk+1)\binom{2k}{k}}\pmod{p^2}
$$
and
$$
\sum_{k=0}^{p-1}\frac{\binom{x}{k}\binom{x+k}{k}(-2)^k}{(dk+1)\binom{2k}{k}}\pmod{p^2},
$$
where $d\in\{0,1,2\}$. As consequences, we extend some known results. For example, for $p>3$ we show
$$
\sum_{k=0}^{p-1}\binom{3k}{k}\left(\frac{4}{27}\right)^k\equiv\frac19+\frac89p+\frac{4}{27}pE_{p-2}\left(\frac13\right)\pmod{p^2},
$$
where $E_n(x)$ denotes the Euler polynomial of degree $n$. This generalizes a known congruence of Z.-W. Sun.
\end{abstract}

\thanks{*Corresponding author}

\maketitle

\section{Introduction}
There is a growing interest in studying supercongruences for sums involving binomial coefficients. Especially, supercongruences involving the central binomial coefficients $\binom{2k}{k}$ were studied widely these years (see e.g., \cite{Guo1,GLS,HH2015JIS,Liu,LR,Mao,MT,Mor,SunZH2014JNT,Sun2011China,Sun2011JNT,VH,ZhangPan}). In 2010, L.-L. Zhao, H. Pan and Z.-W. Sun \cite{ZhaoPanSun2010PAMS} investigated congruences for sums involving $\binom{3k}{k}$ and proved that for any prime $p>5$, one has
$$
\sum_{k=0}^{p-1}\binom{3k}{k}2^k\eq\f{6(-1)^{(p-1)/2}-1}{5}\pmod{p}.
$$
Let $p$ be an odd prime and let $\Z_p$ denote the ring of all $p$-adic integers. In \cite{Sun2009arxiv}, for $p>3,\ x\in \Z_p$ and $d\in\Z$, Z.-W. Sun further studied $\sum_{k=0}^{p-1}\binom{3k}{k+d}x^k\pmod{p}$. In particular, he obtained
\begin{align}
\sum_{k=0}^{p-1}\binom{3k}{k}\l(\f{4}{27}\r)^k&\eq\f19\pmod{p},\label{zwsunres1}\\
\sum_{k=0}^{p-1}\binom{3k}{k+1}\l(\f{4}{27}\r)^k&\eq-\f{16}{9}\pmod{p},\label{zwsunres2}\\
\sum_{k=1}^{p-1}\binom{3k}{k-1}\l(\f{4}{27}\r)^k&\eq-\f{4}{9}\pmod{p}.\label{zwsunres3}
\end{align}
For $k\in\{0,1,\ldots,p-1\}$, it is easy to see that $p\nmid\binom{3k}{k}$ if and only if $0\leq k\leq p/3$ and $p/2<k\leq 2p/3$, and $p\nmid\binom{4k}{2k}$ if and only if $0\leq k< p/4$ and $p/2<k< 3p/4$. Inspired by these work, Z.-H. Sun \cite{SunZH2014IJNT,SunZH2015TJM,SunZH2016IJNT} systematically studied congruences for $\sum_{k=0}^{[p/3]}\binom{3k}{k}x^k$, $\sum_{k=0}^{[p/4]}\binom{4k}{2k}x^k$ and $\sum_{k=(p+1)/2}^{[3p/4]}\binom{4k}{2k}x^k$ modulo $p$, where $[a]$ denotes the integral part of $a$ and $x$ is a $p$-adic integer with $x\not\eq0\pmod p$. In 2015, Kh. Hessami Pilehrood and T. Hessami Pilehrood \cite{HH2015JIS} further investigated congruences for sums involving $\binom{3k}{k}$, $\binom{4k}{2k}$ and the sequence (cf. \cite[A176898]{Sloane})
$$
S_k=\f{\binom{6k}{3k}\binom{3k}{k}}{2(2k+1)\binom{2k}{k}},\quad k=0,1,2,\ldots.
$$
It is easy to see that
\begin{gather*}
\binom{3k}{k}=\f{\binom{-1/3}{k}\binom{-1/3+k}{k}(-27)^k}{\binom{2k}{k}},\quad \binom{4k}{2k}=\f{\binom{-1/4}{k}\binom{-1/4+k}{k}(-64)^k}{\binom{2k}{k}},\\
\f{\binom{6k}{3k}\binom{3k}{k}}{\binom{2k}{k}}=\f{\binom{-1/6}{k}\binom{-1/6+k}{k}(-432)^k}{\binom{2k}{k}},
\end{gather*}
where
$$\binom{x}{k}=\f{x(x-1)\cdots(x-k+1)}{k!}\quad (x\in\R,\ k\in\N=\{0,1,2,\ldots\})$$
are (generalized) binomial coefficients.

Motivated by the above work, in this paper, we study supercongruences for sums involving $\binom{x}{k}\binom{x+k}{k}/\binom{2k}{k}$. These supercongruences are concerned with the Euler polynomials $E_n(x)\ (n\in\N)$ defined by
$$
\f{2\e^{xz}}{\e^z+1}=\sum_{n=0}^{\infty}E_n(x)\f{z^n}{n!}\ (|z|<\pi).
$$
Equivalently,
$$
E_n(x)=\sum_{k=0}^n\binom{n}{k}\f{E_k}{2^k}\l(x-\f12\r)^n,
$$
where $E_0,E_1,\ldots,E_n$ are Euler numbers defined by
$$
E_0=1,\ \t{and}\ \sum_{\substack{k=0\\2\mid n-k}}^n\binom{n}{k}E_k=0\ \t{for}\ n=1,2,3,\ldots.
$$
The reader is referred to \cite{MOS} for some basic properties of the Euler polynomials.

Throughout the paper, for any odd prime $p$ and $x\in\Z_p$, we always use $\<x\>_p$ to denote the least nonnegative residue of $x$ modulo $p$. Write $x=\<x\>_p+pt$, where $t\in\Z_p$.

\begin{theorem}\label{2k+1-4}
Let $p$ be an odd prime and let $x$ be a $p$-adic integer. Then
\begin{equation}\label{2k+1-4eq}
(2x+1)\sum_{k=0}^{p-1}\f{\binom{x}{k}\binom{x+k}{k}(-4)^k}{(2k+1)\binom{2k}{k}}\eq(-1)^{\<x\>_p}(2t+1)-2p t(t+1)E_{p-2}(-x)\pmod{p^2}.
\end{equation}
\end{theorem}

\begin{Rem}
When $x=-\f12$,
$$
\sum_{k=0}^{p-1}\f{\binom{-\f12}{k}\binom{-\f12+k}{k}(-4)^k}{(2k+1)\binom{2k}{k}}=\sum_{k=0}^{p-1}\f{\binom{2k}{k}}{(2k+1)4^k}.
$$
By \cite[Theorem 1.1]{Sun2011JNT}, for $p>3$, we have
$$
p\sum_{k=0}^{p-1}\f{\binom{2k}{k}}{(2k+1)4^k}\eq (-1)^{\f{p-1}{2}}+p^2E_{p-3}\pmod{p^3}.
$$
\end{Rem}

Taking $x=-1/4,-1/3,-1/6$ in Theorem \ref{2k+1-4} we have the following consequences.
\begin{corollary}\label{2k+1-4cor1}
Let $p$ be an odd prime. Then
\begin{equation}\label{2k+1-4cor1eq1}
\sum_{k=0}^{p-1}\f{\binom{4k}{2k}}{(2k+1)16^k}\eq\l(\f{2}{p}\r)+\f34pE_{p-2}\l(\f14\r)\pmod{p^2},
\end{equation}
where $(\f{\cdot}{p})$ stands for the Legendre symbol.

If $p>3$, then we have
\begin{align}
\sum_{k=0}^{p-1}\f{\binom{3k}{k}}{2k+1}\l(\f{4}{27}\r)^k&\eq1+\f43pE_{p-2}\l(\f13\r)\pmod{p^2},\label{2k+1-4cor1eq2}\\
\sum_{k=0}^{p-1}\f{\binom{6k}{3k}\binom{3k}{k}}{(2k+1)108^k\binom{2k}{k}}&\eq\l(\f{3}{p}\r)+\f{5}{12}pE_{p-2}\l(\f16\r)\pmod{p^2}.\label{2k+1-4cor1eq3}
\end{align}
\end{corollary}

\begin{Rem}
The modulus $p$ cases of \eqref{2k+1-4cor1eq1}--\eqref{2k+1-4cor1eq3} were proved by Kh. Hessami Pilehrood and T. Hessami Pilehrood \cite[Corollaries 5, 15 and 34]{HH2015JIS} in 2015.
\end{Rem}

\begin{corollary}\label{-4k+1-4}
Let $p$ be an odd prime and let $x$ be a $p$-adic integer. Then
\begin{equation}\label{-4eq}
\sum_{k=0}^{p-1}\f{\binom{x}{k}\binom{x+k}{k}(-4)^k}{\binom{2k}{k}}\eq(-1)^{\<x\>_p}(2t+1)(2x+1)-4p t(t+1)-2p t(t+1)(2x+1)E_{p-2}(-x)\pmod{p^2}.
\end{equation}
Moreover, if $x\not\eq 0,-1\pmod{p}$, then we have
\begin{equation}\label{k+1-4eq}
2x(x+1)\sum_{k=0}^{p-2}\f{\binom{x}{k}\binom{x+k}{k}(-4)^k}{(k+1)\binom{2k}{k}}\eq\sum_{k=0}^{p-1}\f{\binom{x}{k}\binom{x+k}{k}(-4)^k}{\binom{2k}{k}}-1-\f{2p t(t+1)}{x(x+1)}\pmod{p^2}.
\end{equation}
\end{corollary}

\begin{Rem}
For all $p$-adic integers $x\not\eq0,-1\pmod{p}$,
$$
\f{\binom{x}{p-1}\binom{x+p-1}{p-1}(-4)^{p-1}}{p\binom{2p-2}{p-1}}
$$
are $p$-adic integers, and we can evaluate these terms modulo $p^2$. However, the results are complicated. Therefore, in \eqref{k+1-4eq} we consider the sums over $k$ from $0$ to $p-2$ instead of the ones over $k$ from $0$ to $p-1$.
\end{Rem}

Letting $x=-1/4,-1/3,-1/6$ in Theorem \ref{2k+1-4} we have the following corollaries.

\begin{corollary}\label{-4k+1-4cor}
Let $p$ be an odd prime. Then
\begin{equation}\label{-4k+1-4coreq1}
\sum_{k=0}^{p-1}\f{\binom{4k}{2k}}{16^k}\eq\f14\l(\f{2}{p}\r)+\f34p+\f3{16}pE_{p-2}\l(\f14\r)\pmod{p^2}.
\end{equation}
Moreover, if $p>3$, then we have
\begin{align}
\sum_{k=0}^{p-2}\f{\binom{4k}{2k}}{(k+1)16^k}&\eq\f83-\f23\l(\f{2}{p}\r)+\f{10}{3}p-\f12pE_{p-2}\l(\f14\r)\pmod{p^2},\label{-4k+1-4coreq2}\\
\sum_{k=0}^{p-1}\binom{3k}{k}\l(\f{4}{27}\r)^k&\eq\f19+\f89p+\f{4}{27}pE_{p-2}\l(\f13\r)\pmod{p^2},\label{-4k+1-4coreq3}\\
\sum_{k=0}^{p-2}\f{\binom{3k}{k}}{k+1}\l(\f{4}{27}\r)^k&\eq2+\f52p-\f{1}{3}pE_{p-2}\l(\f13\r)\pmod{p^2},\label{-4k+1-4coreq4}\\
\sum_{k=0}^{p-1}\f{\binom{6k}{3k}\binom{3k}{k}}{108^k\binom{2k}{k}}&\eq\f49\l(\f{3}{p}\r)+\f59p+\f{5}{27}pE_{p-2}\l(\f16\r)\pmod{p^2}.\label{-4k+1-4coreq5}
\end{align}
If $p>5$, then we have
\begin{equation}\label{-4k+1-4coreq6}
\sum_{k=0}^{p-2}\f{\binom{6k}{3k}\binom{3k}{k}}{(k+1)108^k\binom{2k}{k}}\eq\f{18}{5}-\f85\l(\f{3}{p}\r)+\f{26}{5}p-\f{2}{3}pE_{p-2}\l(\f16\r)\pmod{p^2}.
\end{equation}
\end{corollary}

\begin{proof}
Since the proof can be proceed as the argument of Corollary \ref{2k+1-4cor1}, we omit it.
\end{proof}

\begin{Rem}
The modulus $p$ cases of \eqref{-4k+1-4coreq1} and $\eqref{-4k+1-4coreq5}$ were proved by Kh. Hessami Pilehrood and T. Hessami Pilehrood. \cite[Corollaries 5 and 34]{HH2015JIS}. \eqref{-4k+1-4coreq3} extends Z.-W. Sun's result \eqref{zwsunres1} to the modulus $p^2$ case. Note that
$$
\binom{3k}{k+1}=2\l(1-\f{1}{k+1}\r)\binom{3k}{k}\quad \t{and}\quad \binom{3k}{k-1}=\f12\l(1-\f{1}{2k+1}\r)\binom{3k}{k}.
$$
Therefore, via some combinations of Corollaries \ref{2k+1-4} and \ref{-4k+1-4cor}, we can also obtain the modulus $p^2$ extensions of \eqref{zwsunres2} and \eqref{zwsunres3}.
\end{Rem}

\begin{theorem}\label{-2}
Let $p$ be an odd prime and let $x$ be a $p$-adic integer. Then
\begin{align}\label{-2eq}
\sum_{k=0}^{p-1}\f{\binom{x}{k}\binom{x+k}{k}(-2)^k}{\binom{2k}{k}}\eq\ &(-1)^{[(\<x\>_p+1)/2]}\l(1+t-(-1)^{\<x\>_p}\l(\f{-1}{p}\r)t\r)\notag\\
&\ -\f{p t(t+1)}2\l(E_{p-2}\l(\f{x+1}{2}\r)+E_{p-2}\l(-\f{x}{2}\r)\r)\pmod{p^2}.
\end{align}
\end{theorem}

Putting $x=-1/4,-1/3,-1/6$ in Theorem \ref{-2}, we obtain the following results.
\begin{corollary}\label{-2cor} Let $p$ be an odd prime. Then
\begin{align}\label{-2coreq1}
\sum_{k=0}^{p-1}\f{\binom{4k}{2k}}{32^k}\eq&\ \f{3}{32}p\l(E_{p-2}\l(\f38\r)+E_{p-2}\l(\f18\r)\r)\notag\\
&+\begin{cases}\l(\f{-2}{p}\r)(-1)^{[p/8]}\pmod{p},\quad&\t{if}\ p\eq\pm1\pmod8,\\ \f12\l(\f{-2}{p}\r)(-1)^{[p/8]}\pmod{p},\quad&\t{if}\ p\eq\pm3\pmod8.\end{cases}
\end{align}
Moreover, if $p>3$, then we have
\begin{align}
\sum_{k=0}^{p-1}\binom{3k}{k}\l(\f{2}{27}\r)^k&\eq\f13+\f23\l(\f{3}{p}\r)+\f p9\l(E_{p-2}\l(\f13\r)+E_{p-2}\l(\f16\r)\r)\pmod{p^2},\label{-2coreq2}\\
\sum_{k=0}^{p-1}\f{\binom{6k}{3k}\binom{3k}{k}}{216^k\binom{2k}{k}}&\eq\l(\f{6}{p}\r)+\f{5p}{72}\l(E_{p-2}\l(\f5{12}\r)+E_{p-2}\l(\f1{12}\r)\r)\pmod{p^2}.\label{-2coreq3}
\end{align}
\end{corollary}

\begin{proof}
Since the proof can be proceed as the argument of Corollary \ref{2k+1-4cor1}, we overleap it.
\end{proof}

\begin{Rem}
Corollary \ref{-2cor} in the modulus $p$ case was given by Kh. Hessami Pilehrood and T. Hessami Pilehrood. \cite[Corollaries 5, 15, and 34]{HH2015JIS}.
\end{Rem}

\begin{theorem}\label{2k+1-2}
Let $p$ be an odd prime and let $x$ be a $p$-adic integer.
\begin{align}\label{2k+1-2eq}
(2x+1)\sum_{k=0}^{p-1}\f{\binom{x}{k}\binom{x+k}{k}(-2)^k}{(2k+1)\binom{2k}{k}}\eq\ &(-1)^{[\<x\>_p/2]}\l(1+t+(-1)^{\<x\>_p}\l(\f{-1}{p}\r)t\r)\notag\\
&\ -\f{p t(t+1)}2\l(E_{p-2}\l(\f{x+1}{2}\r)+E_{p-2}\l(-\f{x}{2}\r)\r)\pmod{p^2}.
\end{align}
\end{theorem}

Taking $x=-1/4,-1/3,-1/6$ in Theorem \ref{2k+1-2}, we get the following congruences.

\begin{corollary}\label{2k+1-2cor}
Let $p$ be an odd prime. Then
\begin{align}\label{2k+1-2coreq1}
\sum_{k=0}^{p-1}\f{\binom{4k}{2k}}{(2k+1)32^k}\eq&\ \f{3}{16}p\l(E_{p-2}\l(\f18\r)-E_{p-2}\l(\f38\r)\r)\notag\\
&+\begin{cases}\l(\f{-1}{p}\r)(-1)^{[p/8]}\pmod{p},\quad&\t{if}\ p\eq\pm1\pmod8,\\ 2\l(\f{-1}{p}\r)(-1)^{[p/8]}\pmod{p},\quad&\t{if}\ p\eq\pm3\pmod8.\end{cases}
\end{align}
Moreover, for $p>3$ we have
\begin{align}
\sum_{k=0}^{p-1}\f{\binom{3k}{k}}{2k+1}\l(\f{2}{27}\r)^k&\eq-1+2\l(\f{3}{p}\r)+\f p3\l(E_{p-2}\l(\f16\r)-E_{p-2}\l(\f13\r)\r)\pmod{p^2},\label{2k+1-2coreq2}\\
\sum_{k=0}^{p-1}\f{\binom{6k}{3k}\binom{3k}{k}}{(2k+1)216^k\binom{2k}{k}}&\eq\l(\f{2}{p}\r)+\f{5p}{48}\l(E_{p-2}\l(\f1{12}\r)-E_{p-2}\l(\f5{12}\r)\r)\pmod{p^2}.\label{2k+1-2coreq3}
\end{align}
\end{corollary}

\begin{proof}
Since the proof can be proceed as the argument of Corollary \ref{2k+1-4cor1}, we overleap it.
\end{proof}

\begin{Rem}
Corollary \ref{2k+1-2cor} in the modulus $p$ case was given by Kh. Hessami Pilehrood and T. Hessami Pilehrood. \cite[Corollaries 5, 15, and 34]{HH2015JIS}.
\end{Rem}

We shall prove Theorem \ref{2k+1-4} and its corollaries in the next section. Theorems \ref{-2} and \ref{2k+1-2} will be shown in Sections 3 and 4, respectively.

\medskip

\section{Proofs of Theorem \ref{2k+1-4} and Corollaries \ref{2k+1-4cor1} and \ref{-4k+1-4}}
For $n\in\N$ and $x\in\C$, set
$$
F_n(x)=\sum_{k=0}^n\f{(2x+1)\binom{x}{k}\binom{x+k}{k}(-4)^k}{(2k+1)\binom{2k}{k}}.
$$
\begin{lemma}\label{2k+1-4id}
For $n\in\N$ and $x\in\C$, we have
\begin{equation}\label{2k+1-4ideq}
F_n(x)+F_n(x+1)=\f{(-1)^n4^{n+1}}{(2n+1)\binom{2n}{n}}(x+n+1)\binom{x}{n}\binom{x+n}{n}.
\end{equation}
\end{lemma}

\begin{proof}
Denote the left-hand side of \eqref{2k+1-4ideq} by $S_n$ and the right-hand side of \eqref{2k+1-4ideq} by $T_n$. For $n\geq1$, it is easy to see that
\begin{align*}
S_n-S_{n-1}&=F_{n}(x)-F_{n-1}(x)+F_n(x+1)-F_{n-1}(x+1)\\
&=\f{(2x+1)\binom{x}{n}\binom{x+n}{n}(-4)^n}{(2n+1)\binom{2n}{n}}+\f{(2x+3)\binom{x+1}{n}\binom{x+1+n}{n}(-4)^n}{(2n+1)\binom{2n}{n}}\\
&=\l(2x+1+(2x+3)\f{x+1+n}{x+1-n}\r)\f{\binom{x}{n}\binom{x+n}{n}(-4)^n}{(2n+1)\binom{2n}{n}}\\
&=\f{4x^2+8x+2n+4}{x+1-n}\f{\binom{x}{n}\binom{x+n}{n}(-4)^n}{(2n+1)\binom{2n}{n}}
\end{align*}
and
\begin{align*}
T_n-T_{n-1}&=\l(4x+4n+4+\f{2n(2n+1)}{x+1-n}\r)\f{\binom{x}{n}\binom{x+n}{n}(-4)^n}{(2n+1)\binom{2n}{n}}\\
&=\f{4x^2+8x+2n+4}{x+1-n}\f{\binom{x}{n}\binom{x+n}{n}(-4)^n}{(2n+1)\binom{2n}{n}}.
\end{align*}
Clearly, $S_0=T_0=4x+4$. Therefore, $S_n=T_n$ for $n\in\N$. This concludes the proof.
\end{proof}

\begin{lemma}\label{2kk}
For any odd prime $p$, we have
\begin{align}
p\sum_{k=1}^{p-1}\f{4^k}{k\binom{2k}{k}}&\eq-2+2p-4pq_p(2)\pmod{p^2},\label{MTeq1}\\
p\sum_{k=1}^{p-1}\f{4^k}{k^2\binom{2k}{k}}&\eq-4q_p(2)-2pq_p(2)^2\pmod{p^2},\label{MTeq2}\\
p\sum_{k=1}^{p-1}\f{4^k}{(2k+1)\binom{2k}{k}}&\eq0\pmod{p^2},\label{MTcor}
\end{align}
where $q_p(a)=(a^{p-1}-1)/p$ is the Fermat quotient for any $p$-adic integer $a$ with $a\not\eq0\pmod{p}$.
\end{lemma}

\begin{proof}
For $p=3$, one can directly check these congruences. Now we assume $p>3$. Taking $t=4$ in \cite[Theorem 6.1]{MT}, we immediately obtain \eqref{MTeq1} and \eqref{MTeq2}. Clearly,
\begin{align*}
p\sum_{k=1}^{p-1}\f{4^k}{(2k+1)\binom{2k}{k}}&=\f12p\sum_{k=1}^{p-1}\f{\cdot4^{k+1}}{(k+1)\binom{2k+2}{k+1}}=\f12p\sum_{k=2}^p\f{4^k}{k\binom{2k}{k}}\\
&=\f12p\sum_{k=1}^{p-1}\f{4^k}{k\binom{2k}{k}}+\f12\f{4^p}{\binom{2p}{p}}-p.
\end{align*}
With the help of \eqref{MTeq1} and the fact $\binom{2p}{p}\eq2\pmod{p^2}$ (cf. e.g., \cite[p. 380]{Robert00}), we arrive at \eqref{MTcor}.
\end{proof}

\begin{Rem}
Z.-W. Sun \cite[Conjecture 1.1]{Sun2011China} conjectured the modulus $p^3$ extension of \eqref{MTeq2}, and later this conjecture was confirmed by S. Mattarei and R. Tauraso \cite{MT}.
\end{Rem}

\begin{lemma}\label{Fpt}
For any odd prime $p$, we have
$$
F_{p-1}(pt)\eq 2t+1+4p t(t+1)q_p(2)\eq 2t+1-2p t(t+1)E_{p-2}(-pt)\pmod{p^2}.
$$
\end{lemma}

\begin{proof}
It is easy to see that for $k\in\{1,2,\ldots,p-1\}$,
\begin{align}\label{ptkptk}
\binom{pt}{k}\binom{pt+k}{k}&=\f{pt}{pt-k}\binom{pt-1}{k}\binom{pt+k}{k}=\f{pt}{pt-k}\prod_{j=1}^k\l(\f{p^2t^2}{j^2}-1\r)\notag\\
&\eq\f{(-1)^kpt}{pt-k}\eq(-1)^{k-1}\l(\f{pt}{k}+\f{p^2t^2}{k^2}\r)\pmod{p^3}.
\end{align}
Hence
\begin{align*}
F_{p-1}(pt)&=\sum_{k=0}^{p-1}\f{(2p t+1)\binom{pt}{k}\binom{pt+k}{k}(-4)^k}{(2k+1)\binom{2k}{k}}\\
&\eq1+2p t+\sum_{k=1}^{p-1}\f{(2p t+1)4^k}{(2k+1)\binom{2k}{k}}\l(-\f{pt}{k}-\f{p^2t^2}{k^2}\r)\\
&\eq1+2p t -(pt+2p^2t^2)\sum_{k=1}^{p-1}\f{4^k}{(2k+1)k\binom{2k}{k}}-p^2t^2\sum_{k=1}^{p-1}\f{4^k}{(2k+1)k^2\binom{2k}{k}}\pmod{p^2}.
\end{align*}
By Lemma \ref{2kk}, we have
\begin{align*}
p\sum_{k=1}^{p-1}\f{4^k}{(2k+1)k\binom{2k}{k}}=p\sum_{k=1}^{p-1}\f{4^k}{k\binom{2k}{k}}-2p\sum_{k=1}^{p-1}\f{4^k}{(2k+1)\binom{2k}{k}}\eq-2+2p-4pq_p(2)\pmod{p^2}
\end{align*}
and
\begin{align*}
p\sum_{k=1}^{p-1}\f{4^k}{(2k+1)k^2\binom{2k}{k}}&=p\sum_{k=1}^{p-1}\f{4^k}{k^2\binom{2k}{k}}-2p\sum_{k=1}^{p-1}\f{4^k}{k\binom{2k}{k}}+4p\sum_{k=1}^{p-1}\f{4^k}{(2k+1)\binom{2k}{k}}\\
&\eq4-4q_p(2)-4p-2pq_p(2)^2\pmod{p^2}.
\end{align*}
Combining the above, we get
\begin{align*}
F_{p-1}(pt)&\eq 1+2p t-(t+2p t^2)(-2+2p-4pq_p(2))-pt^2(4-4q_p(2)-4p-2pq_p(2)^2)\\
&\eq 2t+1+4p t(t+1)q_p(2)\pmod{p^2}.
\end{align*}
This proves the first congruence. From \cite{MOS} and the fact
$$
pB_{p-1}\eq p-1\pmod{p}
$$
due to L. Carlitz \cite{Carlitz}, we know
\begin{equation}\label{Ep-20}
E_{p-2}(pt)\eq E_{p-2}(0)=\f{2(1-2^{p-1})B_{p-1}}{p-1}\eq-2q_p(2)\pmod{p},
\end{equation}
where $B_n$ is the Bernoulli number. This proves the second congruence.
\end{proof}

\begin{lemma}[Z.-H. Sun {\cite[Lemma 4.2]{SunZH2014JNT}}]\label{ZHSunmpt}
Let $p$ be and off prime, $m\in\{1,2,\ldots,p-1\}$ and $s\in\Z_p$. Then
$$
\binom{m+ps-1}{p-1}\eq\f{ps}{m}-\f{p^2s^2}{m^2}+\f{p^2s}{m}H_m\pmod{p^3},
$$
where $H_m=\sum_{k=1}^m1/k$ denotes the $m$th harmonic number.
\end{lemma}

\begin{lemma}[Z.-H. Sun {\cite[Theorem 2.1]{SunZH2008Discrete}}]\label{ZHSunsum}
Let $n,m\in\N$ and $r,s\in\Z$ with $s\geq0$. Then
$$
\sum_{\substack{k=0\\ k\eq r\pmod m}}^{n-1}(-1)^{\f{k-r}{m}}k^s=-\f{m^s}{2}\l((-1)^{[\f{r-n}{m}]}E_s\l(\f{n}{m}+\l\{\f{r-n}{m}\r\}\r)-(-1)^{[\f r m]}E_s\l(\l\{\f{r}{m}\r\}\r)\r),
$$
where $\{a\}$ stands for the fractional part of $a$.
\end{lemma}

\medskip

\noindent{\it Proof of Theorem \ref{2k+1-4}}. If $\<x\>_p=0$, then $x=pt$. By Lemma \ref{Fpt}, we obtain \eqref{2k+1-4eq}.

Below we assume $\<x\>_p\neq0$. In view of Lemma \ref{2k+1-4id} with $n=p-1$ and Lemma \ref{ZHSunmpt}, for any $p$-adic integer $x$ we have
\begin{align*}
&F_{p-1}(x)-(-1)^{\<x\>_p}F_{p-1}(pt)\\
=& \sum_{k=0}^{\<x\>_p-1}(-1)^k(F_{p-1}(x-k-1)+F_{p-1}(x-k))\\
=&\ \f{(-1)^{p-1}4^p}{(2p-1)\binom{2p-2}{p-1}}\sum_{k=0}^{\<x\>_p-1}(-1)^k(x-k+p-1)\binom{x-k-1}{p-1}\binom{x-k+p-2}{p-1}\\
=&\ \f{(-1)^{p-1}4^p}{(2p-1)\binom{2p-2}{p-1}}\sum_{k=0}^{\<x\>_p-1}(-1)^k(x-k)\binom{x-k-1}{p-1}\binom{x-k+p-1}{p-1}\\
=&\ \f{(-1)^{p-1}4^p}{(2p-1)\binom{2p-2}{p-1}}\sum_{k=0}^{\<x\>_p-1}(-1)^k(\<x\>_p-k+pt)\binom{\<x\>_p-k+pt-1}{p-1}\binom{\<x\>_p-k+p(t+1)-1}{p-1}\\
\eq&\ \f{(-1)^{p-1}4^p}{(2p-1)\binom{2p-2}{p-1}}\sum_{k=0}^{\<x\>_p-1}(-1)^k(\<x\>_p-k)\f{p^2t(t+1)}{(\<x\>_p-k)^2}\\
\eq&\ 4p t(t+1)\sum_{k=0}^{\<x\>_p-1}\f{(-1)^k}{\<x\>_p-k}\pmod{p^2},
\end{align*}
where in the last step we have used Fermat's little theorem and the facts that $(2p-1)\binom{2p-2}{p-1}=p\binom{2p}{p}/2$ and $\binom{2p}{p}\pmod{p^2}$. Putting $n=\<x\>_p+1,\ m=1,\ r=0,\ s=p-2$ in Lemma \ref{ZHSunsum}, we deduce that
\begin{align*}
\sum_{k=1}^{\<x\>_p}\f{(-1)^k}{k}&\eq\sum_{k=0}^{\<x\>_p+1-1}(-1)^kk^{p-2}\\
&=-\f12\l((-1)^{\<x\>_p+1}E_{p-2}(\<x\>_p+1)-E_{p-2}(0)\r)\pmod{p}.
\end{align*}
It is well-known (cf. \cite{MOS}) that $E_n(1-x)=(-1)^nE_n(x)$. Therefore, by \eqref{Ep-20} we have
 $$
 \sum_{k=1}^{\<x\>_p}\f{(-1)^k}{k}\eq-q_p(2)+\f12(-1)^{\<x\>_p+1}E_{p-2}(-x)\pmod{p}.
 $$
So we have
\begin{align}\label{th1key}
&\ F_{p-1}(x)-(-1)^{\<x\>_p}F_{p-1}(pt)\notag\\
\eq&\ 4p t(t+1)4p t(t+1)\sum_{k=0}^{\<x\>_p-1}\f{(-1)^k}{\<x\>_p-k}\notag\\
=&\ (-1)^{\<x\>_p}4p t(t+1)\sum_{k=1}^{\<x\>_p}\f{(-1)^k}{k}\notag\\
\eq&\ (-1)^{\<x\>_p}4p t(t+1)\l(-q_p(2)+\f12(-1)^{\<x\>_p+1}E_{p-2}(-x)\r)\pmod{p^2}.
\end{align}
Combining this with Lemma \ref{Fpt}, we arrive at
$$
F_{p-1}(x)\eq(-1)^{\<x\>_p}(2t+1)-2p t(t+1)E_{p-2}(-x)\pmod{p^2}.
$$
We are done.\qed

\medskip

\noindent{\it Proof of Corollary \ref{2k+1-4cor1}}. Putting $x=-1/4$ in \eqref{2k+1-4eq}, we obtain
$$
\f12\sum_{k=0}^{p-1}\f{\binom{4k}{2k}}{(2k+1)16^k}\eq(-1)^{\<-1/4\>_p}(2t+1)-2p t(t+1)E_{p-2}\l(\f14\r)\pmod{p^2}.
$$
Note that
$$
\l\<-\f14\r\>_p=\begin{cases}(p-1)/4,\quad&\t{if}\ p\eq1\pmod4,\\ (3p-1)/4,\quad&\t{if}\ p\eq3\pmod4\end{cases}
$$
and
$$
t=\f{-1/4-\<-1/4\>_p}{p}=\begin{cases}-1/4,\quad&\t{if}\ p\eq1\pmod4,\\ -3/4,\quad&\t{if}\ p\eq3\pmod4.\end{cases}
$$
Therefore,
$$
\sum_{k=0}^{p-1}\f{\binom{4k}{2k}}{(2k+1)16^k}\eq\begin{cases}(-1)^{(p-1)/4}+3pE_{p-2}\l(\f14\r)/4,\quad&\t{if}\ p\eq1\pmod4,\\
(-1)^{(p+1)/4}+3pE_{p-2}\l(\f14\r)/4,\quad&\t{if}\ p\eq3\pmod4,
\end{cases}
$$
i.e.,
$$
\sum_{k=0}^{p-1}\f{\binom{4k}{2k}}{(2k+1)16^k}\eq\begin{cases}1+3pE_{p-2}\l(\f14\r)/4,\quad&\t{if}\ p\eq\pm1\pmod8,\\
-1+3pE_{p-2}\l(\f14\r)/4,\quad&\t{if}\ p\eq\pm3\pmod8.
\end{cases}
$$
It is well-known (cf. \cite[p. 57]{IR}) that
$$
\l(\f{2}{p}\r)=\begin{cases}1,\quad&\t{if}\ p\eq\pm1\pmod{8},\\-1,\quad&\t{if}\ p\eq\pm3\pmod{8}.\end{cases}
$$
Combining the above, we obtain \eqref{2k+1-4cor1eq1}.

Substituting $x=-1/3$ into \eqref{2k+1-4eq}, we get
$$
\f13\sum_{k=0}^{p-1}\f{\binom{3k}{k}}{2k+1}\l(\f{4}{27}\r)^k\eq(-1)^{\<-1/3\>_p}(2t+1)-2p t(t+1)E_{p-2}\l(\f13\r)\pmod{p^2},
$$
where
$$
\l\<-\f13\r\>_p=\begin{cases}(p-1)/3\eq0\pmod2,\quad&\t{if}\ p\eq1\pmod3,\\ (2p-1)/3\eq1\pmod2,\quad&\t{if}\ p\eq2\pmod3\end{cases}
$$
and
$$
t=\f{-1/3-\<-1/3\>_p}{p}=\begin{cases}-1/3,\quad&\t{if}\ p\eq1\pmod3,\\ -2/3,\quad&\t{if}\ p\eq2\pmod3.\end{cases}
$$
This proves \eqref{2k+1-4cor1eq2}.

Taking $x=-1/6$ in \eqref{2k+1-4eq}, we have
$$
\f23\sum_{k=0}^{p-1}\f{\binom{6k}{3k}\binom{3k}{k}}{(2k+1)108^k\binom{2k}{k}}\eq(-1)^{\<-1/6\>_p}(2t+1)-2p t(t+1)E_{p-2}\l(\f16\r)\pmod{p^2}.
$$
Now
$$
\l\<-\f16\r\>_p=\begin{cases}(p-1)/6,\quad&\t{if}\ p\eq1\pmod6,\\ (5p-1)/6,\quad&\t{if}\ p\eq5\pmod6\end{cases}
$$
and
$$
t=\f{-1/6-\<-1/6\>_p}{p}=\begin{cases}-1/6,\quad&\t{if}\ p\eq1\pmod6,\\ -5/6,\quad&\t{if}\ p\eq5\pmod6.\end{cases}
$$
It suffices to show
$$
\l(\f3p\r)=\begin{cases}(-1)^{(p-1)/6},\quad&\t{if}\ p\eq1\pmod6,\\(-1)^{5(p+1)/6},\quad&\t{if}\ p\eq5\pmod6.\end{cases}
$$
In fact, by the law of quadratic reciprocity (cf. \cite[p. 53]{IR}), we have
$$
\l(\f3p\r)=(-1)^{(p-1)/2}\l(\f p3\r)=\begin{cases}(-1)^{(p-1)/2}=(-1)^{(p-1)/6},\quad&\t{if}\ p\eq1\pmod6,\\(-1)^{(p+1)/2}=(-1)^{5(p+1)/6},\quad&\t{if}\ p\eq5\pmod6.\end{cases}
$$
This proves \eqref{2k+1-4cor1eq3}.

The proof of Corollary \ref{2k+1-4cor1} is now complete.\qed

\begin{lemma}\label{-4k+1-4id}
For any nonnegative integer $n$ and complex number $x$, we have
\begin{align}\label{-4k+1-4ideq1}
&\sum_{k=0}^n\f{\binom{x}{k}\binom{x+k}{k}(-4)^k}{\binom{2k}{k}}-(2x+1)^2\sum_{k=0}^n\f{\binom{x}{k}\binom{x+k}{k}(-4)^k}{(2k+1)\binom{2k}{k}}\notag\\
=&\ \f{(-1)^n4^{n+1}(n-x)(n+1+x)\binom{x}{n}\binom{x+n}{n}}{(2n+1)\binom{2n}{n}}
\end{align}
and
\begin{align}\label{-4k+1-4ideq2}
&\sum_{k=0}^n\f{\binom{x}{k}\binom{x+k}{k}(-4)^k}{\binom{2k}{k}}-2x(x+1)\sum_{k=0}^n\f{\binom{x}{k}\binom{x+k}{k}(-4)^k}{(k+1)\binom{2k}{k}}\notag\\
=&\ 1+\f{(-1)^n2^{2n+1}(n-x)(n+1+x)\binom{x}{n}\binom{x+n}{n}}{(n+1)\binom{2n}{n}}.
\end{align}
\end{lemma}

\begin{proof}
Here we omit the proofs of Lemma \ref{-4k+1-4id}, since they are quite similar to the one of Lemma \ref{2k+1-4id}
\end{proof}

\medskip

\noindent{\it Proof of Corollary \ref{-4k+1-4}}. We first consider \eqref{-4eq}. By \eqref{-4k+1-4ideq1} with $n=p-1$, we have
\begin{align*}
&\sum_{k=0}^{p-1}\f{\binom{x}{k}\binom{x+k}{k}(-4)^k}{\binom{2k}{k}}\\
=&\ (2x+1)^2\sum_{k=0}^{p-1}\f{\binom{x}{k}\binom{x+k}{k}(-4)^k}{(2k+1)\binom{2k}{k}}+\f{4^p(p-1-x)(p+x)\binom{x}{p-1}\binom{x+p-1}{p-1}}{(2p-1)\binom{2p-2}{p-1}}.
\end{align*}
In view of \eqref{2k+1-4eq}, it suffices to prove
\begin{equation}\label{-4k+1-4pf1}
\f{4^p(p-1-x)(p+x)\binom{x}{p-1}\binom{x+p-1}{p-1}}{(2p-1)\binom{2p-2}{p-1}}\eq-4p t(t+1)\pmod{p^2}.
\end{equation}
If $\<x\>_p=0$, then $x=pt$. Now, by Lemma \ref{ZHSunmpt} we have
\begin{align*}
&\f{4^p(p-1-pt)(p+pt)\binom{pt}{p-1}\binom{pt+p-1}{p-1}}{(2p-1)\binom{2p-2}{p-1}}\\
=&\ \f{4^p(p-1-pt)(1+pt)\binom{pt}{p-1}\binom{p(t+1)}{p-1}}{(2p-1)\binom{2p-2}{p-1}}\\
\eq&\ \f{4^p(p-1-pt)(1+pt)p^2t(t+1)}{(2p-1)\binom{2p-2}{p-1}}\\
\eq&\ -4p t(t+1)\pmod{p^2}.
\end{align*}
Suppose that $\<x\>_p\neq0$. With the help of Lemma \ref{ZHSunmpt}, we obtain
\begin{align*}
&\f{4^p(p-1-x)(p+x)\binom{x}{p-1}\binom{x+p-1}{p-1}}{(2p-1)\binom{2p-2}{p-1}}\\
=&\ -\f{4^p(p+x)x\binom{x-1}{p-1}\binom{x+p-1}{p-1}}{(2p-1)\binom{2p-2}{p-1}}\\
=&\ -\f{4^p(p+\<x\>_p+pt)(\<x\>_p+pt)\binom{\<x\>_p+pt-1}{p-1}\binom{\<x\>_p+p(t+1)-1}{p-1}}{(2p-1)\binom{2p-2}{p-1}}\\
\eq&\ -\f{4^p(p+\<x\>_p+pt)(\<x\>_p+pt)p^2t(t+1)}{(2p-1)\binom{2p-2}{p-1}\<x\>_p^2}\\
\eq&\ -4p t(t+1)\pmod{p^2}.
\end{align*}
Therefore \eqref{-4k+1-4pf1} holds and this proves \eqref{-4eq}.

Below we consider \eqref{k+1-4eq}. Putting $n=p-1$ in \eqref{-4k+1-4ideq2}, we have
\begin{align*}
&\sum_{k=0}^{p-1}\f{\binom{x}{k}\binom{x+k}{k}(-4)^k}{\binom{2k}{k}}-2x(x+1)\sum_{k=0}^{p-1}\f{\binom{x}{k}\binom{x+k}{k}(-4)^k}{(k+1)\binom{2k}{k}}\\
=&\ 1+\f{4^p(p-1-x)(p+x)\binom{x}{p-1}\binom{x+p-1}{p-1}}{2p\binom{2p-2}{p-1}}.
\end{align*}
Thus it suffices to show
\begin{equation*}
\f{2x(x+1)\binom{x}{p-1}\binom{x+p-1}{p-1}(-4)^{p-1}}{p\binom{2p-2}{p-1}}+\f{4^p(p-1-x)(p+x)\binom{x}{p-1}\binom{x+p-1}{p-1}}{2p\binom{2p-2}{p-1}}\eq \f{2p t(t+1)}{x(x+1)}\pmod{p^2},
\end{equation*}
i.e.,
\begin{equation}\label{-4k+1-4pf2}
-\f{4^p\binom{x}{p-1}\binom{x+p-1}{p-1}}{2\binom{2p-2}{p-1}}\eq\f{2p t(t+1)}{x(x+1)}\pmod{p^2}.
\end{equation}
In fact, by Lemma \ref{ZHSunmpt},
\begin{align*}
-\f{4^p\binom{x}{p-1}\binom{x+p-1}{p-1}}{2\binom{2p-2}{p-1}}&=-\f{4^p\binom{\<x\>_p+pt}{p-1}\binom{\<x\>_p+p(t+1)-1}{p-1}}{2\binom{2p-2}{p-1}}\\
&\eq-\f{4^pp^2t(t+1)}{2\binom{2p-2}{p-1}\<x\>_p(\<x\>_p+1)}\\
&\eq\f{2p t(t+1)}{x(x+1)}\pmod{p^2}.
\end{align*}
Therefore \eqref{-4k+1-4pf2} holds and it proves \eqref{k+1-4eq}.

The proof of Corollary \ref{-4k+1-4} is now complete.\qed

\medskip

\section{Proof of Theorem 1.2}
For $n\in\N$ and $x\in\C$, set
$$
G_n(x)=\sum_{k=0}^{n}\f{\binom{x}{k}\binom{x+k}{k}(-2)^k}{\binom{2k}{k}}.
$$

\begin{lemma}\label{th2lem1}
For $n\in\N$ and $x\in\C$, we have
\begin{equation*}
G_n(x)+G_n(x+2)=\f{(-1)^n2^{n+1}\binom{x+1}{n}\binom{x+1+n}{n}}{\binom{2n}{n}}.
\end{equation*}
\end{lemma}

\begin{proof}
Because the proof can be proceed as the argument of Lemma \ref{2k+1-4id} with minor modification, here we overleap it.
\end{proof}

\begin{lemma}[S. Mattarei and R. Tauraso {\cite[p. 155]{MT}}]\label{th2lem2} For any odd prime $p$, we have
\begin{align*}
p\sum_{k=1}^{p-1}\f{2^k}{k\binom{2k}{k}}&\eq\l(\f{-1}{p}\r)-1-pq_p(2)\pmod{p^2},\\
p\sum_{k=1}^{p-1}\f{2^k}{k^2\binom{2k}{k}}&\eq-q_p(2)\pmod{p^2}.
\end{align*}
\end{lemma}

\begin{Rem}
The modulus $p^3$ extensions of the congruences in Lemma \ref{th2lem2} were conjectured by Z.-W. Sun \cite{Sun2011China,Sun2011JNT} and confirmed by  S. Mattarei and R. Tauraso \cite{MT}.
\end{Rem}

\begin{lemma}\label{th2lem3} For any odd prime $p$, we have
\begin{align}
G_{p-1}(pt)&\eq1+t-\l(\f{-1}{p}\r)t+pt(t+1)q_p(2)\pmod{p^2},\label{th2lem3eq1}\\
G_{p-1}(pt-1)&\eq1-t+\l(\f{-1}{p}\r)t+pt(t-1)q_p(2)\pmod{p^2}.\label{th2lem3eq2}
\end{align}
\end{lemma}

\begin{proof}
We first prove \eqref{th2lem3eq1}. In light of \eqref{ptkptk},
\begin{align*}
G_{p-1}(pt)&=\sum_{k=0}^{p-1}\f{\binom{pt}{k}\binom{pt+k}{k}(-2)^k}{\binom{2k}{k}}\\
&\eq1-\sum_{k=1}^{p-1}\f{2^k}{\binom{2k}{k}}\l(\f{pt}{k}+\f{p^2t^2}{k^2}\r)\\
&=1-pt\sum_{k=1}^{p-1}\f{2^k}{k\binom{2k}{k}}-p^2t^2\sum_{k=1}^{p-1}\f{2^k}{k^2\binom{2k}{k}}\pmod{p^2}.
\end{align*}
Then \eqref{th2lem3eq1} follows from Lemma \ref{th2lem2} at once.

We now prove \eqref{th2lem3eq2}. From the proof of Lemma \ref{th2lem2}, we have for $k\in\{1,2,\ldots,p-1\}$,
\begin{equation}\label{pt-1k}
\binom{pt-1}{k}\binom{pt-1+k}{k}=\f{pt}{pt+k}\binom{pt-1}{k}\binom{pt+k}{k}\eq(-1)^k\l(\f{pt}{k}-\f{p^2t^2}{k^2}\r)\pmod{p^3}.
\end{equation}
Therefore,
\begin{align*}
G_{p-1}(pt-1)&=\sum_{k=0}^{p-1}\f{\binom{pt-1}{k}\binom{pt-1+k}{k}(-2)^k}{\binom{2k}{k}}\\
&\eq1+pt\sum_{k=1}^{p-1}\f{2^k}{k\binom{2k}{k}}-p^2t^2\sum_{k=1}^{p-1}\f{2^k}{k^2\binom{2k}{k}}\pmod{p^2}.
\end{align*}
Making use of Lemma \ref{th2lem2}, we obtain \eqref{th2lem3eq2}.
\end{proof}

\medskip

\noindent{\it Proof of Theorem \ref{-2}}. We divide the proof into two cases according to the parity of $\<x\>_p$.

\medskip

{\bf Case 1.} $\<x\>_p$ is even.

If $\<x\>_p=0$, then $x=pt$. Combining \eqref{th2lem3eq1} and \eqref{Ep-20}, we obtain \eqref{-2eq}.

Now we suppose that $\<x\>_p\neq0$. In view of Lemma \ref{th2lem1} with $n=p-1$ and Lemma \ref{ZHSunmpt}, we achieve that
\begin{align}\label{-2key1}
&G_{p-1}(x)-(-1)^{\<x\>_p/2}G_{p-1}(pt)\notag\\
=&\sum_{k=0}^{(\<x\>_p-2)/2}(-1)^k(G_{p-1}(x-2k-2)+G_{p-1}(x-2k))\notag\\
=&\ \f{2^p}{\binom{2p-2}{p-1}}\sum_{k=0}^{(\<x\>_p-2)/2}(-1)^k\binom{\<x\>_p-2k+pt-1}{p-1}\binom{\<x\>_p-2k-1+p(t+1)-1}{p-1}\notag\\
\eq&\ \f{2^pp^2t(t+1)}{\binom{2p-2}{p-1}}\sum_{k=0}^{(\<x\>_p-2)/2}\f{(-1)^k}{(\<x\>_p-2k)(\<x\>_p-2k-1)}\notag\\
\eq&\ -2(-1)^{\<x\>_p/2}pt(t+1)\sum_{k=1}^{\<x\>_p/2}\f{(-1)^k}{2k(2k-1)}\notag\\
=&\ -2(-1)^{\<x\>_p/2}pt(t+1)\l(\sum_{k=1}^{\<x\>_p/2}\f{(-1)^k}{2k-1}-\f12\sum_{k=1}^{\<x\>_p/2}\f{(-1)^k}{k}\r)\pmod{p^2}.
\end{align}
By Lemma \ref{ZHSunsum} with $n=\<x\>_p/2+1,\ m=1,\ r=1,\ s=p-2$, we get
\begin{align}\label{-2key1-1}
\sum_{k=1}^{\<x\>_p/2}\f{(-1)^k}{k}&\eq\sum_{k=0}^{\<x\>_p/2}(-1)^k k^{p-2}\notag\\
&=-\f12\l((-1)^{(\<x\>_p+2)/2}E_{p-2}\l(\f{\<x\>_p+2}{2}\r)-E_{p-2}(0)\r)\notag\\
&\eq-\f12\l((-1)^{\<x\>_p/2}E_{p-2}\l(-\f{\<x\>_p}{2}\r)+2q_p(2)\r)\notag\\
&\eq-q_p(2)-\f12(-1)^{\<x\>_p/2}E_{p-2}\l(-\f{x}{2}\r)\pmod{p}.
\end{align}
Similarly, putting $n=\<x\>_p,\ m=2,\ r=1,\ s=p-2$ in Lemma \ref{ZHSunsum}, we have
\begin{align}\label{-2key1-2}
\sum_{k=1}^{\<x\>_p/2}\f{(-1)^k}{2k-1}&\eq-\sum_{\substack{k=0\\ k\eq1\pmod2}}^{\<x\>_p-1}(-1)^{(k-1)/2} k^{p-2}\notag\\
&=\f{2^{p-2}}{2}\l((-1)^{\<x\>_p/2}E_{p-2}\l(\f{\<x\>_p+1}{2}\r)-E_{p-2}\l(\f12\r)\r)\notag\\
&\eq\ \f14(-1)^{\<x\>_p/2}E_{p-2}\l(\f{x+1}{2}\r)\pmod{p},
\end{align}
where in the last step we have used the fact $E_n(1/2)=0$ for any positive odd integer $n$. Substituting \eqref{-2key1-1} and \eqref{-2key1-2} into \eqref{-2key1}, we arrive at
\begin{align*}
&G_{p-1}(x)-(-1)^{\<x\>_p/2}G_{p-1}(pt)\\
\eq&\ -(-1)^{\<x\>_p/2}pt(t+1)q_p(2)-\f{p t(t+1)}{2}\l(E_{p-2}\l(\f{x+1}{2}\r)+E_{p-2}\l(-\f x2\r)\r)\pmod{p^2}.
\end{align*}
This, together with \eqref{th2lem3eq1} gives \eqref{-2eq}.

\medskip

{\bf Case 2.} $\<x\>_p$ is odd.

In view of Lemma \ref{th2lem1} with $n=p-1$, we obtain
\begin{align}\label{-2key2}
&G_{p-1}(x)-(-1)^{(\<x\>_p+1)/2}G_{p-1}(pt-1)\notag\\
=&\sum_{k=0}^{(\<x\>_p-1)/2}(-1)^k(G_{p-1}(x-2k-2)+G_{p-1}(x-2k))\notag\\
=&\f{2^p}{\binom{2p-2}{p-1}}\sum_{k=0}^{(\<x\>_p-1)/2}(-1)^k\binom{\<x\>_p-2k+pt-1}{p-1}\binom{\<x\>_p-2k-1+p(t+1)-1}{p-1}\notag\\
=&(-1)^{(\<x\>_p-1)/2}\f{2^p}{\binom{2p-2}{p-1}}\binom{pt}{p-1}\binom{pt+p-1}{p-1}\notag\\
&+\f{2^p}{\binom{2p-2}{p-1}}\sum_{k=0}^{(\<x\>_p-3)/2}(-1)^k\binom{\<x\>_p-2k+pt-1}{p-1}\binom{\<x\>_p-2k-1+p(t+1)-1}{p-1}.
\end{align}
By \eqref{ptkptk}, we have
\begin{align}\label{-2key2-1}
&(-1)^{(\<x\>_p-1)/2}\f{2^p}{\binom{2p-2}{p-1}}\binom{pt}{p-1}\binom{pt+p-1}{p-1}\notag\\
\eq&\ (-1)^{(\<x\>_p+1)/2}\f{2^p}{\binom{2p-2}{p-1}}\l(\f{pt}{p-1}+\f{p^2t^2}{(p-1)^2}\r)\notag\\
=&\ (-1)^{(\<x\>_p+1)/2}\f{2^{p+1}(2p-1)}{\binom{2p}{p}}\l(\f{t}{p-1}+\f{pt^2}{(p-1)^2}\r)\notag\\
\eq&\ (-1)^{(\<x\>_p+1)/2}\l(2t-2p t(t+1)+2p tq_p(2)\r)\pmod{p^2}.
\end{align}
Moreover, by Lemma \ref{ZHSunmpt} we get
\begin{align}
&\f{2^p}{\binom{2p-2}{p-1}}\sum_{k=0}^{(\<x\>_p-3)/2}(-1)^k\binom{\<x\>_p-2k+pt-1}{p-1}\binom{\<x\>_p-2k-1+p(t+1)-1}{p-1}\notag\\
\eq&\ \f{2^pp^2t(t+1)}{\binom{2p-2}{p-1}}\sum_{k=0}^{(\<x\>_p-3)/2}\f{(-1)^k}{(\<x\>_p-2k)(\<x\>_p-2k-1)}\notag\\
\eq&\ (-1)^{(\<x\>_p+1)/2} 2p t(t+1)\sum_{k=1}^{(\<x\>_p-1)/2}\f{(-1)^k}{2k(2k+1)}\notag\\
=&\ (-1)^{(\<x\>_p+1)/2} 2p t(t+1)\l(\f12\sum_{k=1}^{(\<x\>_p-1)/2}\f{(-1)^k}{k}-\sum_{k=1}^{(\<x\>_p-1)/2}\f{(-1)^k}{2k+1}\r)\pmod{p^2}\label{-2key2-2}.
\end{align}
Via similar arguments of \eqref{-2key1-1} and \eqref{-2key1-2}, we have
\begin{align}
\sum_{k=1}^{(\<x\>_p-1)/2}\f{(-1)^k}{k}&\eq-q_p(2)-\f12(-1)^{(\<x\>_p+1)/2}E_{p-2}\l(\f{x+1}{2}\r)\pmod{p},\label{-2key2-3}\\
\sum_{k=1}^{(\<x\>_p-1)/2}\f{(-1)^k}{2k+1}&\eq\f14(-1)^{(\<x\>_p+1)/2}E_{p-2}\l(-\f{x}{2}\r)-1\pmod{p}.\label{-2key2-4}
\end{align}
Combining \eqref{-2key2-1}--\eqref{-2key2-4}, we arrive at
\begin{align*}
&G_{p-1}(x)-(-1)^{(\<x\>_p+1)/2}G_{p-1}(pt-1)\\
\eq&(-1)^{(\<x\>_p+1)/2}\l(2t-pt(t-1)q_p(2)-\f{p t(t+1)}2\l(E_{p-2}\l(\f{x+1}{2}\r)+E_{p-2}\l(-\f{x}{2}\r)\r)\r)\pmod{p^2}.
\end{align*}
This, together with \eqref{th2lem3eq2}, gives \eqref{-2eq}.

In view of the above, the proof of Theorem \ref{-2} is now complete.\qed

\medskip

\section{Proof of Theorem \ref{2k+1-2}}

For $n\in\N$ and $x\in\C$, let
$$
H_n(x)=\sum_{k=0}^n\f{(2x+1)\binom{x}{k}\binom{x+k}{k}(-2)^k}{(2k+1)\binom{2k}{k}}.
$$

\begin{lemma}\label{th3lem1}
For $n\in\N$ and $x\in\C$, we have
$$
H_{n}(x)+H_n(x+2)=\f{(-1)^n2^{n+1}(2x+3)\binom{x+1}{n}\binom{x+1+n}{n}}{(2n+1)\binom{2n}{n}}.
$$
\end{lemma}

\begin{proof}
We omit the proof since it can be proceed similarly to the argument of Lemma \ref{2k+1-4id} with minor modification.
\end{proof}

\begin{lemma}\label{th3lem2}
For any odd prime $p$, we have
$$
p\sum_{k=1}^{p-1}\f{2^k}{(2k+1)\binom{2k}{k}}\eq\l(\f{-1}{p}\r)-p\pmod{p^2}.
$$
\end{lemma}

\begin{proof}
Note that
$$
p\sum_{k=1}^{p-1}\f{2^k}{(2k+1)\binom{2k}{k}}=p\sum_{k=1}^{p-1}\f{2^{k+1}}{(k+1)\binom{2k+2}{k+1}}=p\sum_{k=1}^{p-1}\f{2^k}{k\binom{2k}{k}}+\f{2^p}{\binom{2p}{p}}-p.
$$
Then we obtain the desired result by Lemma \ref{th2lem2} and the fact $\binom{2p}{p}\eq2\pmod{p^2}$.
\end{proof}

\begin{lemma}\label{th3lem3}
For any odd prime $p$, we have
\begin{align}
H_{p-1}(pt)&\eq1+t+\l(\f{-1}{p}\r)t+pt(t+1)q_p(2)\pmod{p^2},\label{th3lem3eq1}\\
H_{p-1}(pt-1)&\eq-1+t+\l(\f{-1}{p}\r)t-pt(t-1)q_p(2)\pmod{p^2}.\label{th3lem3eq2}
\end{align}
\end{lemma}

\begin{proof}
We first consider \eqref{th3lem3eq1}. In view of \eqref{ptkptk},
\begin{align*}
H_{p-1}(pt)&=\sum_{k=0}^{p-1}\f{(2p t+1)\binom{pt}{k}\binom{pt+k}{k}(-2)^k}{(2k+1)\binom{2k}{k}}\\
&\eq1+2p t-(1+2p t)\sum_{k=1}^{p-1}\f{2^k}{(2k+1)\binom{2k}{k}}\l(\f{pt}{k}+\f{p^2t^2}{k^2}\r)\\
&\eq1+2p t-pt\sum_{k=1}^{p-1}\f{2^k}{k\binom{2k}{k}}+2p t\sum_{k=1}^{p-1}\f{2^k}{(2k+1)\binom{2k}{k}}-p^2t^2\sum_{k=1}^{p-1}\f{2^k}{k^2\binom{2k}{k}}\pmod{p^2}.
\end{align*}
With the help of Lemmas \ref{th2lem2} and \ref{th3lem2}, we obtain \eqref{th3lem3eq1}.

We now consider \eqref{th3lem3eq2}. By \eqref{pt-1k},
\begin{align*}
H_{p-1}(pt-1)&=\sum_{k=0}^{p-1}\f{(2p t-1)\binom{pt-1}{k}\binom{pt-1+k}{k}(-2)^k}{(2k+1)\binom{2k}{k}}\\
&\eq-1+2p t+(2p t-1)\sum_{k=1}^{p-1}\f{2^k}{(2k+1)\binom{2k}{k}}\l(\f{pt}{k}-\f{p^2t^2}{k^2}\r)\\
&\eq-1+2p t-pt\sum_{k=1}^{p-1}\f{2^k}{k\binom{2k}{k}}+2p t\sum_{k=1}^{p-1}\f{2^k}{(2k+1)\binom{2k}{k}}+p^2t^2\sum_{k=1}^{p-1}\f{2^k}{k^2\binom{2k}{k}}\pmod{p^2}.
\end{align*}
By Lemmas \ref{th2lem2} and \ref{th3lem2}, we get \eqref{th3lem3eq2}.
\end{proof}

\medskip

\noindent{\it Proof of Theorem \ref{2k+1-2}}. We divide the proof into two cases according to the parity of $\<x\>_p$.

\medskip

{\bf Case 1.} $\<x\>_p$ is even.

If $\<x\>_p=0$, then $x=pt$. Combining \eqref{th3lem3eq1} and \eqref{Ep-20}, we obtain \eqref{2k+1-2eq}.

Suppose that $\<x\>_p\neq0$. By Lemma \ref{th3lem1} with $n=p-1$ and Lemma \ref{ZHSunmpt}, we get
\begin{align}\label{2k+1-2key1}
&H_{p-1}(x)-(-1)^{\<x\>_p/2}H_{p-1}(pt)\notag\\
=&\sum_{k=0}^{(\<x\>_p-2)/2}(-1)^k(H_{p-1}(x-2k-2)+H_{p-1}(x-2k))\notag\\
=&\ \f{2^p}{(2p-1)\binom{2p-2}{p-1}}\sum_{k=0}^{(\<x\>_p-2)/2}(-1)^k(2x-4k-1)\binom{\<x\>_p-2k+pt-1}{p-1}\binom{\<x\>_p-2k-1+p(t+1)-1}{p-1}\notag\\
\eq&\ \f{2^pp^2t(t+1)}{(2p-1)\binom{2p-2}{p-1}}\sum_{k=0}^{(\<x\>_p-2)/2}\f{(-1)^k(2\<x\>_p-4k-1)}{(\<x\>_p-2k)(\<x\>_p-2k-1)}\notag\\
\eq&\ 2(-1)^{\<x\>_p/2}pt(t+1)\sum_{k=1}^{\<x\>_p/2}\f{(-1)^k(4k-1)}{2k(2k-1)}\notag\\
=&\ 2(-1)^{\<x\>_p/2}pt(t+1)\l(\f12\sum_{k=1}^{\<x\>_p/2}\f{(-1)^k}{k}+\sum_{k=1}^{\<x\>_p/2}\f{(-1)^k}{2k-1}\r)\pmod{p^2}.
\end{align}
Substituting \eqref{-2key1-1} and \eqref{-2key1-2} into \eqref{2k+1-2key1}, we arrive at
\begin{align*}
&H_{p-1}(x)-(-1)^{\<x\>_p/2}H_{p-1}(pt)\\
\eq&\ -(-1)^{\<x\>_p/2}pt(t+1)q_p(2)-\f{pt(t+1)}2\l(E_{p-2}\l(-\f x2\r)-E_{p-2}\l(\f{x+1}{2}\r)\r)\pmod{p^2}.
\end{align*}
This, together with \eqref{th3lem3eq1} gives \eqref{2k+1-2eq}.

\medskip

{\bf Case 2.} $\<x\>_p$ is odd.

In view of Lemma \ref{th3lem1} with $n=p-1$, we obtain
\begin{align}\label{2k+1-2key2}
&H_{p-1}(x)-(-1)^{(\<x\>_p+1)/2}H_{p-1}(pt-1)\notag\\
=&\sum_{k=0}^{(\<x\>_p-1)/2}(-1)^k(H_{p-1}(x-2k-2)+H_{p-1}(x-2k))\notag\\
=&\f{2^p}{(2p-1)\binom{2p-2}{p-1}}\sum_{k=0}^{(\<x\>_p-1)/2}(-1)^k(2x-4k-1)\binom{\<x\>_p-2k+pt-1}{p-1}\binom{\<x\>_p-2k-1+p(t+1)-1}{p-1}\notag\\
=&(-1)^{(\<x\>_p-1)/2}\f{2^p(2p t+1)}{(2p-1)\binom{2p-2}{p-1}}\binom{pt}{p-1}\binom{pt+p-1}{p-1}\notag\\
&+\f{2^p}{(2p-1)\binom{2p-2}{p-1}}\sum_{k=0}^{(\<x\>_p-3)/2}(-1)^k(2x-4k-1)\binom{\<x\>_p-2k+pt-1}{p-1}\binom{\<x\>_p-2k-1+p(t+1)-1}{p-1}.
\end{align}
By \eqref{ptkptk}, we have
\begin{align}\label{2k+1-2key2-1}
&(-1)^{(\<x\>_p-1)/2}\f{2^p(2p t+1)}{(2p-1)\binom{2p-2}{p-1}}\binom{pt}{p-1}\binom{pt+p-1}{p-1}\notag\\
\eq&\ (-1)^{(\<x\>_p+1)/2}\f{2^p(2p t+1)}{(2p-1)\binom{2p-2}{p-1}}\l(\f{pt}{p-1}+\f{p^2t^2}{(p-1)^2}\r)\notag\\
=&\ (-1)^{(\<x\>_p+1)/2}\f{2^{p+1}(2 p t+1)}{\binom{2p}{p}}\l(\f{t}{p-1}+\f{pt^2}{(p-1)^2}\r)\notag\\
\eq&\ (-1)^{(\<x\>_p-1)/2}\l(2t+2p t(t+1)+2p tq_p(2)\r)\pmod{p^2}.
\end{align}
Moreover, by Lemma \ref{ZHSunmpt} we get
\begin{align}
&\f{2^p}{(2p-1)\binom{2p-2}{p-1}}\sum_{k=0}^{(\<x\>_p-3)/2}(-1)^k(2x-4k-1)\binom{\<x\>_p-2k+pt-1}{p-1}\binom{\<x\>_p-2k-1+p(t+1)-1}{p-1}\notag\\
\eq&\ \f{2^pp^2t(t+1)}{(2p-1)\binom{2p-2}{p-1}}\sum_{k=0}^{(\<x\>_p-3)/2}\f{(-1)^k(2\<x\>_p-4k-1)}{(\<x\>_p-2k)(\<x\>_p-2k-1)}\notag\\
\eq&\ (-1)^{(\<x\>_p-1)/2} 2p t(t+1)\sum_{k=1}^{(\<x\>_p-1)/2}\f{(-1)^k(4k+1)}{2k(2k+1)}\notag\\
=&\ (-1)^{(\<x\>_p-1)/2} 2p t(t+1)\l(\f12\sum_{k=1}^{(\<x\>_p-1)/2}\f{(-1)^k}{k}+\sum_{k=1}^{(\<x\>_p-1)/2}\f{(-1)^k}{2k+1}\r)\pmod{p^2}\label{2k+1-2key2-2}.
\end{align}
Combining \eqref{2k+1-2key2}--\eqref{2k+1-2key2-2}, \eqref{-2key2-3} and \eqref{-2key2-4}, we arrive at
\begin{align*}
&H_{p-1}(x)-(-1)^{(\<x\>_p+1)/2}H_{p-1}(pt-1)\\
\eq&(-1)^{(\<x\>_p-1)/2}\l(2t-pt(t-1)q_p(2)-\f{p t(t+1)}2\l(E_{p-2}\l(-\f{x}{2}\r)-E_{p-2}\l(\f{x+1}{2}\r)\r)\r)\pmod{p^2}.
\end{align*}
This, together with \eqref{th3lem3eq2}, gives \eqref{2k+1-2eq}.

In view of the above, the proof of Theorem \ref{2k+1-2} is now complete.\qed

\begin{Ack}
This work is supported by the National Natural Science Foundation of China (grant 12201301).
\end{Ack}


\begin{thebibliography}{99}
\bibitem{Carlitz} L. Carlitz, Some congruences for the Bernoulli numbers, Amer. J. Math. 75 (1953), 163--172.

\bibitem{Guo1} V.J.W. Guo, Some generalizations of a supercongruence of van Hamme, Integral Transforms Spec. Funct. 28 (2017), 888--899.

\bibitem{GLS} V.J.W. Guo, J.-C. Liu and M.J. Schlosser, An extension of a supercongruence of Long and Ramakrishna, Proc. Amer. Math. Soc. 151 (2023), 1157--1166.

\bibitem{HH2015JIS} Kh. Hessami Pilehrood and T. Hessami Pilehrood, Jacobi polynomials and congruences involving some higher-order Catalan numbers and Binomial coefficients, J. Integer Seq. 18 (2015), Art. 15.11.7.

\bibitem{IR} K. Ireland and M. Rosen, A Classical Introduction to Modern Number Theory, 2nd ed., Graduate Texts in Mathematics, Vol. 84, Springer, New York, 1990.

\bibitem{Liu} J.-C. Liu, Proof of some divisibility results on sums involving binomial coefficients, J. Number Theory 180 (2017), 566--572.

\bibitem{LR} L. Long and R. Ramakrishna, Some supercongruences occurring in truncated hypergeometric series,
 Adv. Math. 290 (2016), 773--808.

\bibitem{MOS} W. Magnus, F. Oberhettinger and R.P. Soni, Formulas and Theorems for the Special Functions of Mathematical Physics (3rd edition), Springer, New York, 1966.

\bibitem{Mao} G.-S. Mao, Proof of a conjecture of Adamchuk, J. Combin. Theory Ser. A 182 (2021), Art. 105478.

\bibitem{MT} S. Mattarei and R. Tauraso, Congruences for central binomial sums and finite polylogarithms, J. Number Theory 133 (2013), 131--157.

\bibitem{Mor} E. Mortenson, A $p$-adic supercongruence conjecture of van Hamme, Proc. Amer. Math. Soc. 136 (2008), 4321--4328.

\bibitem{Robert00} A.M. Robert, A Course in $p$-Adic Analysis, Graduate Texts in Mathematics, Vol. 198, Springer-Verlag, New York, 2000.

\bibitem{Sloane} N.J.A. Sloane, The On-Line Encyclopedia of Integer Sequences, http://oeis.org.

\bibitem{SunZH2008Discrete} Z.-H. Sun, Congruences involving Bernoulli polynomials, Discrete Math. 308 (2008), no. 1, 71--112.

\bibitem{SunZH2014IJNT} Z.-H. Sun, Congruences concerning Lucas sequences, Int. J. Number Theory 10 (2014), no. 3, 793--815.

\bibitem{SunZH2014JNT} Z.-H. Sun, Generalized Legendre polynomials and related supercongruences, J. Number Theory 143 (2014), 293--319.

\bibitem{SunZH2015TJM} Z.-H. Sun, Quartic residues and sums involving $\binom{4k}{2k}$, Taiwanese J. Math. 19 (2015), no. 3, 803--818.

\bibitem{SunZH2016IJNT} Z.-H. Sun, Cubic congruences and sums involving $\binom{3k}{k}$, Int. J. Number Theory 12 (2016), no. 1, 143--164.

\bibitem{Sun2009arxiv} Z.-W. Sun, Various congruences involving binomial coeffcients and higher-order catalan numbers, preprint, arXiv:0909.3808, 2009.

\bibitem{Sun2011China} Z.-W. Sun, Super congruences and Euler numbers, Sci. China Math. 54 (2011), no. 12, 2509--2535.

\bibitem{Sun2011JNT} Z.-W. Sun, On congruences related to central binomial coefficients, J. Number Theory 131 (2011), no. 11, 2219--2238.

\bibitem{VH} L. Van Hamme, Some conjectures concerning partial sums of generalized hypergeometric series, in: $p$-Adic Functional Analysis (Nijmegen, 1996), Lecture Notes in Pure and Appl. Math. 192. Dekker, NewYork, 1997, 223--236.

\bibitem{ZhangPan} Y. Zhang and H. Pan, Some $3$-adic congruences for binomial sums, 57 (2014), 711--718.

\bibitem{ZhaoPanSun2010PAMS} L.-L. Zhao, H. Pan and Z.-W. Sun, Some congruences for the second-order Catalan numbers, Proc. Amer. Math. Soc. 138 (2010), no. 1, 37--46.
\end{thebibliography}
\end{document}